\documentclass[a4paper,11pt]{amsart}

\usepackage[english]{babel}
\usepackage[utf8]{inputenc}
\usepackage{paralist,url,verbatim}
\usepackage{amscd}
\usepackage[T1]{fontenc}
\usepackage{xspace}
\usepackage{comment}
\usepackage{amsmath,amsthm,amssymb,amsfonts}
\usepackage[bookmarksopen=true]{hyperref}

\theoremstyle{plain}
\newtheorem{theorem}{\bf Theorem}

\newtheorem{lemma}[theorem]{Lemma}

\theoremstyle{definition}
\newtheorem{remark}[theorem]{Remark}
\newtheorem{definition}[theorem]{Definition}

\newcommand{\reg}{\operatorname{reg} }
\newcommand{\In}{\operatorname{In} }

\newcommand{\F}{\mathcal{F}}

\newcommand{\D}{\Delta}
\newcommand{\lk}{\operatorname{lk}}

\newcommand{\gorstar}{Gorenstein$^{*}$\xspace}

\newcommand{\excise}[1]{}

\begin{document}

\title{Linear syzygies, flag complexes, and regularity}

\author[A.~Constantinescu]{Alexandru Constantinescu} \address{Mathematisches Institut, Freie Universität
  Berlin, Arnimallee 3, 14195 Berlin, Germany}
\email{aconstant@math.fu-berlin.de}
\urladdr{\href{http://userpage.fu-berlin.de/aconstant/Main.html}{http://userpage.fu-berlin.de/aconstant/Main.html}}

\author[T.~Kahle]{Thomas Kahle} \address {Fakultät für Mathematik,
  OvGU Magdeburg, Universitätsplatz 2, 39106 Magdeburg, Germany}
\email{thomas.kahle@ovgu.de}
\urladdr{\href{http://www.thomas-kahle.de}{http://www.thomas-kahle.de}}

\author[M.~Varbaro]{Matteo Varbaro}
\address{ Dipartimento di Matematica,
 Universit\`a di Genova,
 Via Dodecaneso 35, Genova 16146, Italy}
\email{varbaro@dima.unige.it}
\urladdr{\href{http://www.dima.unige.it/~varbaro/}{http://www.dima.unige.it/~varbaro/}} 



\date{October 2014} 

\begin{abstract}
  We show that for every $r\in\mathbb{Z}_{>0}$ there exist monomial
  ideals generated in degree two, with linear syzygies, and regularity
  of the quotient equal to~$r$.  Such examples can not be found among
  Gorenstein ideals since the regularity of their quotients is at most
  four.  We also show that for most monomial ideals generated in
  degree two and with linear syzygies the regularity is
  $O(\log(\log(n))$, where $n$ is the number of variables.
\end{abstract}

\maketitle

Let $n$ be a positive integer, $S=K[x_1,\ldots ,x_n]$ the polynomial
ring in $n$ variables over a field $K$.
Any quotient $S/I$ by some homogeneous ideal $I\subseteq S$ has a
minimal graded free resolution.  The number of minimal generators of a
given degree of the free modules occurring in the resolution are
independent of the resolution chosen and define the \emph{Betti
  numbers} $\beta_{i,j}(S/I)$.  The \emph{Castelnuovo-Mumford
  regularity} of $S/I$ is $\reg(S/I):=\max\{j-i:\beta_{i,j}\neq 0\}$.

\subsection*{Bounding the regularity}
Equipped with these definitions it is a basic question to understand
extreme values and shapes of the Betti numbers and the modules that
realize them.  One line of research, which we contribute to here, is
to bound the regularity in terms of the number of variables for
specified classes of ideals.  To get interesting bounds one has to put
strong restrictions.  A class of examples due to Mayr and Meyer shows
that even for quadratically generated binomial ideals in $n$
variables, regularity of the order of $2^{2^{n}}$ is
possible~\cite{mayr82:_compl_word_probl_commut_semig_polyn_ideal,bayer1988complexity}.
In view of this result, interest shifted to specific classes of ideals
with good geometric or algebraic properties.  For example, Eisenbud
and Goto conjectured in~\cite{eisenbud1984linear} that if $I$ is a
prime ideal defining a variety of codimension $r$ and degree $d$ in
$\mathbb{P}^{n-1}$, then $\reg(S/I) \leq d-r$.  See
\cite{bayer93:_what} for a broader overview. In other directions, the
defining ideals of Koszul algebras have a good bound: If $S/I$ is
Koszul, then $\reg(S/I)\leq n$. The same bound is also satisfied by
monomial ideals generated in degree two.  In order to get stricter
bounds, one has to impose more restrictions on the class of ideals.
In this note we are interested in quadratically generated ideals
whose resolutions are linear for a few steps:

\begin{definition}
  For any positive integer $p$, the $K$-algebra $S/I$ \emph{satisfies
    property~$N_p$} if:
  \[
  \beta_{i,j}(S/I)=0 \quad \forall \ i\in\{1,\ldots ,p\} \mbox{ and }j\neq
  i+1.
  \]
\end{definition}

If $S/I$ is Koszul and satisfies property~$N_p$, a recent result of
Avramov, Conca, and
Iyengar~\cite[Theorem~6.1]{avramov2013subadditivity} implies that
\[
\reg(S/I)\leq 2\lfloor n/(p+1) \rfloor + 1.
\]
It is currently unknown whether the above bound is sharp.  However,
ideals for which $S/I$ is Koszul, that satisfies property~$N_2$, and
have $\reg(S/I)\sim \sqrt{n}$
exist~\cite[Example~6.9]{avramov2013subadditivity}.
In contrast to the belief that ideals defining Koszul rings and
quadratic monomial ideals should have similar homological properties
(they often do), if $I$ is a monomial ideal such that $S/I$
satisfies~$N_p$, a much better bound has been established by Dao,
Huneke, and Schweig.

\begin{theorem}\label{thm:DHS}\cite{dao2013bounds}
  Let $I\subseteq S=K[x_1,\ldots ,x_n]$ be a monomial ideal such that
  $S/I$ satisfies $N_p$ for some $p\geq 2$. Then
  \[
  \reg(S/I)< \log_{(p+3)/2}\bigg(\frac{2n}{p}\bigg)+2.
  \]
\end{theorem}

Even if the above bound becomes stricter when $p$ grows, its
logarithmic nature already shows for $p=2$.  Thus, in a sense, the
crucial assumption is that $S/I$ satisfies~$N_2$. In other words, when
$I$ is quadratically generated and its syzygy module is linear.  For
this reason, we mainly work in the case where $I$ is a monomial ideal
such that $S/I$ satisfies $N_2$.  Without loss of generality we deal
with square-free monomial ideals, since by polarization we can always
assume this (at most doubling the number of variables).  Such
ideals are known as \emph{edge ideals}: Given a
simple graph $G$ on $n$ vertices, its edge ideal is defined as
\[
I(G)=(x_ix_j:\{i,j\}\mbox{ is an edge of }G)\subseteq S=K[x_1,\ldots
,x_n].
\]
It is often convenient to think of edge ideals as Stanley-Reisner
ideals of flag simplicial complexes.  Precisely, $I(G)=I_{\In(G)}$
where $\In(G)$ is the \emph{independence complex} of~$G$.  By
definition, every flag simplicial complex is of this form.  From the
explicit description of the syzygy module of a monomial ideal, one can
immediately check that
\[
\mbox{$S/I(G)$ satisfies $N_2$ }\iff \mbox{ $\In(G)$ has no induced
  4-cycles}.
\]
More generally, it can be proved that $N_{p}$
is equivalent to $\In(G)$ having no induced $(p+2)$-cycles (cf. \cite[Theorem~2.7]{dao2013bounds}).  If $\D$
is a flag simplicial complex with no induced 4-cycles, we say that
$\D$ is {\it flag-no-square}.

\begin{remark}\label{r:systolic}
  Flag-no-square simplicial complexes are called $5$-large in the
  literature on hyperbolic Coxeter groups, such
  as~\cite{januszkiewicz2003hyperbolic}.  More generally a flag
  complex is \emph{$k$-large} if its \emph{systole}---the shortest
  induced cycle---has length at least~$k$.  Confusingly, in that
  literature an induced subcomplex is called a \emph{full subcomplex}
  (even if then what is arguably an empty cycle goes by the name of
  full cycle).
\end{remark}

Our first result strengthens Theorem~\ref{thm:DHS} in the case that
$S/I$, besides satisfying $N_2$, is also Gorenstein.  Then there is a universal bound for the
regularity.  We learned the averaging argument in
its proof from Davis' book on Coxeter groups~\cite[Lemma 6.11.5]{davis2008geometry}.

\begin{theorem}\label{thm:gorN2}
  Let $I\subseteq S$ be a monomial ideal such that $S/I$ is
  Gorenstein.
\begin{enumerate}[(i)]
\item If $S/I$ satisfies $N_2$, then $\reg(S/I)\leq 4$.
\item If $S/I$ satisfies $N_3$, then $\reg(S/I)\leq 2$.
\end{enumerate}
\end{theorem}
\begin{proof}
  Up to polarization, we can assume that $I=I_{\Delta}$ where $\Delta$
  is a Gorenstein flag simplicial complex. By \cite[Theorem 7]{S77}
  there exist a \gorstar complex $\Gamma$ and a simplex $\tau$
  such that $\D=\tau*\Gamma$.
Thus, without loss of generality, we assume that $\Delta$ is \gorstar
of dimension $d-1$, so that
\[
\dim(K[\Delta])=\reg(K[\Delta])=d.
\]
Let $\sigma$ be a $(d-3)$-dimensional face of $\Delta$. Since
$C=\lk_{\D}\sigma$ is a 1-dimensional flag \gorstar complex, it must
be a $k$-cycle for some $k\geq 4$. Using flagness again, it follows
that $C$ is an induced $k$-cycle of~$\D$.

Let $A$ be the average number of facets of $\D$ containing a given
$(d-3)$-dimensional face.  Since any facet of $\D$ contains exactly
$\binom{d}{2}$ faces of dimension $d-3$, the average number is
\[
A=\frac{f_{d-1}}{f_{d-3}} \binom{d}{2}.
\]
Since $K[\D]$ is Gorenstein of regularity $d$, the Dehn-Sommerville
equations $h_i=h_{d-i}$ hold for $i=0,\ldots ,\delta:=\lfloor
d/2\rfloor$. Let
\begin{align*}
\widehat{h_i} :=\begin{cases}
h_i & \mbox{if }i<d/2 \\
h_i/2 & \mbox{if }i=d/2.
\end{cases}
\end{align*}
From the relation between $f$- and $h$-vector
($f_{j-1}=\sum_{i=0}^j\binom{d-i}{j-i}h_i$ for $j=0,\ldots, d$) we get
\[f_{d-1}=2\cdot \sum_{i=0}^{\delta}\widehat{h_i} \quad
\textup{and}\quad
f_{d-3}=\sum_{i=0}^{\delta}\bigg(\binom{d-i}{2}+\binom{i}{2}\bigg)\widehat{h_i}.\]
Since $ \binom{d-i}{2}+\binom{i}{2}<\binom{d-(i-1)}{2}+\binom{i-1}{2}$
for $i=1,\ldots ,\delta$ we get
$f_{d-3}>\left(\binom{d-\delta}{2}+\binom{\delta}{2}\right)
\sum_{i=0}^{\delta}\widehat{h_i}$. Therefore
\begin{equation*}
A<\frac{2\cdot \binom{d}{2}}{\binom{d-\delta}{2}+\binom{\delta}{2}}.
\end{equation*}
The right hand side of the above inequality evaluates as
\begin{align*}\label{eq:average}
A<\begin{cases}
\frac{4(d-1)}{d-2} & \mbox{if $d$ is even} \\
\frac{4d}{d-1} & \mbox{if $d$ is odd}
\end{cases}
\end{align*}
In particular, if $d>4$, then $A<5$, so there exists a
$(d-3)$-dimensional face $\sigma$ of $\Delta$ such that
$\lk_{\D}\sigma$ is a $4$-cycle. If $d>2$, then $A<6$, so there exists
a $(d-3)$-dimensional face $\sigma$ of $\Delta$ such that
$\lk_{\D}\sigma$ is a $5$-cycle. By the first paragraph of this proof,
such cycles are induced cycles of~$\Delta$, so we get a contradiction
to $S/I$ satisfying, respectively, $N_2$ or $N_3$.
\end{proof}

Among edge ideals with linear syzygies, the example with highest
regularity in~\cite{dao2013bounds} achieved the value four.  This
limitation may be due to the fact that the natural strategy to produce
such examples is to construct a flag-no-square triangulation of a
$d$-sphere for high~$d$.  However, by Theorem~\ref{thm:gorN2}, such
triangulations do not exist whenever $d>3$.  Even more,
Theorem~\ref{thm:gorN2} implies that, in order to find examples of
unbounded regularity, one has to leave the world of manifolds too:

\begin{remark}
  Let $\D$ be a flag-no-square triangulation of a (homology)
  $d$-manifold. Then $d\leq 4$. This follows immediately from Theorem
  \ref{thm:gorN2}, because $\lk_{\D}v$ is a flag-no-square
  triangulation of a (homology) $(d-1)$-sphere.
\end{remark}

Dropping the insistence on manifolds, however, it is possible to find edge
ideals with linear syzygies and arbitrarily high regularity.  We found
them---somewhat surprisingly---related to a question of Gromov on the
existence of hyperbolic Coxeter groups with arbitrarily high virtual
cohomological dimension.  The question was answered positively by
Januszkiewicz and \'Swi\c{a}tkowski
in~\cite{januszkiewicz2003hyperbolic}.  The interesting fact for our
purposes is that, to answer Gromov's question, they built a
flag-no-square {\it closed orientable pseudomanifold} of dimension~$r$
for any positive integer~$r$.  These complexes are denoted~$L_r$
(see~\cite[Section 6]{januszkiewicz2003hyperbolic}).

A
 simplicial complex $\D$ is a {\it closed pseudomanifold} if it is strongly connected (in particular pure) and any codimension 1 face is contained in exactly two facets. A closed pseudomanifold is called {\it orientable} when for any codimension 1 face $F$, if $F\cup\{i\}$ and $F\cup\{j\}$ are the two facets containing it then $|\{k\in F:k<i\}|+|\{k\in F:k<j\}|$ is odd.

\begin{theorem}
  For any integer $r\geq 1$, there exists a graph $G_r$ on $n(r)$
  vertices such that $S/I(G_r)$ satisfies $N_2$, where $
  S=K[x_1,\ldots ,x_{n(r)}]$, and
\[\reg(S/I(G_r))=r.\]
\end{theorem}
\begin{proof}
  Let $G_r$ be the complement of the 1-skeleton of $L_{r-1}$, so that
  $L_{r-1}=\In(G_r)$.  Then $S/I(G_r)$ satisfies~$N_2$. Furthermore,
  because $L_{r-1}$ is a closed orientable pseudomanifold of dimension
  $r-1$, it is straightforward to check that 
 \[\sum_{F \mbox{ {\scriptsize facet of }}L_{r-1}}F\]
 is a top-dimensional cycle. In particular
\[
\widetilde{H}_{r-1}(L_{r-1};K)\neq 0.
\]
By Hochster's
formula~\cite[Corollary~5.12]{miller05:_combin_commut_algeb}),
$\reg(K[L_{r-1}])=r$.
\end{proof}

As noted in~\cite{januszkiewicz2003hyperbolic}, the number $n(r)$ in
the above theorem is huge, growing much more quickly than exponential
in~$r$.  Consequently the family $\{L_r\}_r$ is not suitable to show
that Theorem~\ref{thm:DHS} is (asymptotically) sharp.  In
\cite{januszkiewicz2003hyperbolic} it was also observed that any
family of flag-no-square pseudomanifolds $\D$ of dimension~$d$ is
forced to have a huge number of vertices.

In the remainder of this work we quantify their result and extend its
proof to flag-no-square complexes with no free codimension 1 faces,
i.e. codimension 1 faces contained in only one facet
(Theorem~\ref{thm:double}).  To this end we prove that the number of
vertices of such a simplicial complex is at least doubly exponential
in the dimension.

\begin{lemma}\label{numericallemma}
  For any integer $k\geq 3$,
  \[\prod_{i=0}^{k-3}(k-i)^{2^i}<12^{2^{k-3}}.\]
\end{lemma}
\begin{proof}
  This is a routine computation using the inequality $(i-1)(i+1)<i^2$
  several times.
\end{proof}
\excise{
  \begin{proof}
    We use the fact that $(i-1)(i+1)<i^2$ for any positive
    integer $i$. Notice that, by using this role, we have:
    {\small \[\prod_{i=0}^{k-3}(k-i)^{2^i}<(k-1)^{3}(k-2)^3\prod_{i=3}^{k-3}(k-i)^{2^i}<(k-2)^6(k-3)^5\prod_{i=4}^{k-3}(k-i)^{2^i}<\ldots \]}We
    want to show that we can go this way to the end. To this aim, it
    is enough to show that, for any $j\leq k-5$, we can prove that
    \[\prod_{i=0}^{k-3}(k-i)^{2^i}<(k-j)^{a_j}(k-{j+1})^{2^{j+1}-a_{j-1}}\prod_{i=j+2}^{k-3}(k-i)^{2^i} \]with
    $2^j<a_j<2^{j+1}$. We have this for $a_1=3$ and $a_2=6$. By
    induction, note
    that \[a_j=2^j-a_{j-2}+a_{j-1}<2^j+a_{j-1}<2^j+2^j=2^{j+1}.\]Furthermore,
    since $a_{j-2}<2^{j-1}$ and $a_{j-1}>2^{j-1}$, the equality
    $a_j=2^j-a_{j-2}+a_{j-1}$ also gives $a_j>2^j$.

    By taking $j=k-5$, therefore we get:
    \[\prod_{i=0}^{k-3}(k-i)^{2^i}<5^{a_{k-5}}4^{2^{k-4}-a_{k-6}}3^{2^{k-3}}. \]
    Since $a_{k-5}<2^{k-4}<2^{k-3}$, we can apply our trick still one
    time, by getting:
    \[5^{a_{k-5}}4^{2^{k-4}-a_{k-6}}3^{2^{k-3}}<4^{2^{k-4}-a_{k-6}+a_{k-5}}3^{2^{k-3}-a_{k-4}}<12^{2^{k-3}}.\]
  \end{proof}
}

\begin{theorem}\label{thm:double}
  Let $\Delta$ be a $d$-dimensional flag-no-square simplicial complex
  with no free $(d-1)$-faces. Then, if $(f_{-1},f_0,\ldots ,f_d)$ is
  the $f$-vector of $\D$,
  \[
  f_d> (25/12)^{2^{d-2}} \ \ \ \ \ \mbox{and} \ \ \ \ \ f_0>(25/12)^{2^{d-3}}. 
  \]
\end{theorem}
\begin{proof}
  Let $v_{d}$ and $s_{d}$ denote respectively the minimal $f_0(\D)$
  and $f_d(\D)$ attained in the class of simplicial complexes in the
  statement.  Let $\D$ be such a complex.  The link $\lk_\D v$ of any
  vertex $v\in\Delta$ is a $(d-1)$-dimensional flag-no-square
  simplicial complex with no free $(d-2)$-faces.  By double counting,
  we find
  \[
  f_d(\Delta)=\frac{1}{d+1}\sum_{v\in\D}f_{d-1}(\lk_\D v).
  \]
  Therefore 
  \[
  f_d(\Delta)\geq \frac{f_0(\Delta)\cdot s_{d-1}}{d+1}.
  \]
  Fix a vertex $v\in\Delta$. Any facet of $\lk_\D v$ is a $(d-1)$-face
  of~$\D$.  Since $\D$ has no free $(d-1)$-faces, to any facet of
  $\lk_\D v$, we can associate a vertex of $\D$ which does not belong
  to the star of $v$. In other words we defined a function
  \[
  \phi:\F_{d-1}(\lk_\D v)\longrightarrow \F_0(\D)\setminus
  \F_0(\mathrm{star}_\D v).
  \]
  Because of the no-square condition, and since $\lk_\D v$ is an
  induced subcomplex of~$\D$, one can check that $\phi$ is injective
  and thus
  \[
  f_0(\D)\geq s_{d-1}+ v_{d-1}+1.
  \]
  In particular, putting together the above inequalities we get
  \[
  s_{d}>\frac{ s_{d-1}^2}{d+1} \ \ \ \ \ v_{d}> s_{d-1}.
  \]
  Since $s_{1}=5$, we find
  \[
  f_d(\D)> \frac{5^{2^{d-1}}}{\prod_{i=0}^{d-2}(d+1-i)^{2^{i}}}.
  \]
  Finally, by Lemma~\ref{numericallemma},
  \[
  f_d(\D)> \frac{5^{2^{d-1}}}{12^{2^{d-2}}}=(25/12)^{2^{d-2}}.\qedhere
  \]
\end{proof}

\begin{remark}
  Unfortunately the doubly exponential bound in
  Theorem~\ref{thm:double} can not be easily extended to arbitrary
  flag-no-square simplicial complexes (replacing $d$ with the top
  degree in which the homology does not vanish).  While it is always
  possible to get rid of the free faces by collapses, indeed, this
  operation does not preserve flagness.
  For a general flag-no-square simplicial complex, the exponential
  bound found by Dao, Huneke, and Schweig is to our knowledge the best
  possible.
\end{remark}

\begin{remark}
If a flag-no-square simplicial complex of dimension 2 has a free 1-face $e\subset F$, where $F$ is the only 2-face containing $e$, then if we collapse the pair $(e,F)$ we still get a flag-no-square simplicial complex (this is a peculiarity of the dimension 2 case). This observation, together with the proof of Theorem \ref{thm:double}, yields the following:
Let $I(G)\subseteq S=K[x_1,\ldots ,x_n]$ be an edge ideal such that
$\dim(S/I(G))=\reg(S/I(G))=3$. If $S/I(G)$ satisfies $N_2$, then
$n\geq 12$. If $n=12$, then $\In(G)$ is the boundary of the
icosahedron.
\end{remark}

\bibliographystyle{amsalpha}
\bibliography{flagnosquare}

\end{document}